\documentclass[12pt]{amsart}
\usepackage{amsmath,amsthm,amsfonts,amssymb}

\newtheorem{theorem}{Theorem}[section] 
 
\newtheorem{lemma}[theorem]{Lemma}

\newtheorem{prop}[theorem]{Proposition}
\newtheorem{remark}[theorem]{Remark}

\numberwithin{equation}{section}

\usepackage{color}

\newcommand{\eps}{\varepsilon}

\newcommand{\x}{\mathbf{x}}
\newcommand{\z}{\mathbf{z}}
\newcommand{\e}{\mathbf{e}}

\newcommand{\bu}{\mathbf{u}}
\newcommand{\bk}{\mathbf{k}}
\newcommand{\y}{\mathbf{y}}
\newcommand{\w}{\mathbf{w}}

\newcommand{\0}{\mathbf{0}}

\newcommand{\I}{\mathbb{I}}
\newcommand{\R}{\mathbb{R}}
\newcommand{\Z}{\mathbb{Z}}

\begin{document}

\title[Radius of Elastic Manifolds]{Self-Repelling Elastic Manifolds with Low Dimensional Range}
\author{Carl Mueller}
\address{Carl Mueller
\\Department of Mathematics
\\University of Rochester
\\Rochester, NY  14627}
\email{carl.e.mueller@rochester.edu}
\author{Eyal Neuman}
\address{Eyal Neuman
\\Department of Mathematics
\\Imperial College
\\London, UK}
\email{e.neumann@imperial.ac.uk}
\thanks{CM was partially supported by Simons Foundation grant 513424.} 

\begin{abstract} 
 We consider self-repelling elastic manifolds with a domain 
$[-N,N]^d \cap \mathbb{Z}^d$, that take values in $\mathbb{R}^D$. Our main result states that when the dimension of the domain is $d=2$ and the dimension of the range is $D=1$, the effective radius $R_N$ of the manifold is approximately $N^{4/3}$. This verifies the conjecture of Kantor, Kardar and Nelson \cite{Kardar_87}. 
Our results for the case where $d \geq 3$ and $D <d$ give a lower bound on 
$R_N$ of order $N^{\frac{1}{D} \left(d-\frac{2(d-D)}{D+2} \right)}$ and an 
upper bound proportional to $N^{\frac{d}{2}+\frac{d-D}{D+2}}$. These results 
imply that self-repelling elastic manifolds with a low dimensional range 
undergo a significantly stronger stretching than in the case where 
$d=D$, which was studied in \cite{muel-neum-2021}.   
\end{abstract}

\maketitle

\section{Introduction} \label{section:introduction} Self-repelling elastic 
manifolds were first introduced by Kantor, Kardar and Nelson in \cite{Kardar_86} as generalizations of 
polymer models to higher dimensions, in order to capture the behaviour of sheets of covalently bonded atoms and of polymerized lipid surfaces, among others. See \cite{Kardar_86,Kardar_87,Kantor_87,Nelson_2004} and references therein for additional details. 
  In the mathematical literature this model was first studied in \cite{muel-neum-2021} as a random surface with free boundary conditions is 
modeled by $\mathbb{R}^D$-valued discrete Gaussian free field (DGFF) over 
$[-N,N]^d \cap \mathbb{Z}^d$, with Neumann boundary conditions. A 
penalization term for self-intersections, which reflects the fact that 
different parts of the manifold cannot occupy the same position, is then 
added to the Hamiltonian of the DGFF. 
If the domain of DGFF is one dimensional, then we recover the well-known 
model of a random polymer.  A typical object of study is the end-to-end 
distance of such a polymer, or the closely related concept of effective 
radius.  There is a vast literature on such problems, see Bauerschmidt, 
Duminil-Copin, Goodman, and Slade \cite{BDGS12} and the included citations.  

Bounds on the effective radius $R_N$ of the self-repelling manifold were 
derived in \cite{muel-neum-2021} for the case where $d=D$. It was proved 
that for the two dimensional case, that is $d=D=2$, $R_N$ is proportional to 
$N$ in the upper and lower bounds, up to a logarithmic correction. The 
bounds on $R_N$ in higher dimensions are not as sharp, with a lower bound 
proportional to $N$, but with an upper bound of order $N^{d/2}$. The results 
of \cite{muel-neum-2021} proved however that self-repelling elastic 
manifolds experience a substantial stretching in any dimension.  

We should also mention the related paper \cite{muel-neum-2021a} on the 
effective radius for moving polymers.  Most methods available in the polymer literature 
fail in these more general settings, but we did take inspiration from work 
of Bolthausen \cite{Bol90}.

In this paper, we deal with the case where $D<d$. We first prove that when 
the dimension of the domain is $d=2$ and the dimension of the range is $D=1$, 
the effective radius $R_N$ of the manifold is approximately $N^{4/3}$. Our 
results for the case $d \geq 3$ and $D <d$ give a lower bound on $R_N$ of 
order $N^{\frac{1}{D} \left(d-\frac{2(d-D)}{D+2} \right)}$ and an upper 
bound proportional to $N^{\frac{d}{2}+\frac{d-D}{D+2}}$. These results imply 
that self-repelling elastic manifolds with a low dimensional range undergo a 
significantly stronger stretching than in the case $d=D$, which was 
studied in \cite{muel-neum-2021}.   

The remaining case, $D>d$, looks to be much harder. For example, consider the case where the domain of the self-repelling DGFF is 
$\{0,\dots,N\}$ and it takes values in $\R^D$.  For $D=2,3,4$ the behavior 
of the effective radius of the self-repelling polymer as $N\to\infty$ is 
still unknown,  although we have good information for $D=1$ and for $D>4$.  
See page 400 of Bauerschmidt, Duminil-Copin, Goodman, and Slade \cite{BDGS12} 
and also Bauerschmidt, Slade, and Tomberg, and Wallace \cite{BSTW17}.  
If $D$ is large enough, then for self-avoiding walks, the lace 
expansion can be used.  For DGFF however, there appears to be no analogue of the lace 
expansion.  
 \section{Model Setup and Main Results} 
 \subsection{Setup}
 We briefly review some of the definitions and notation from Section 1 of \cite{muel-neum-2021} which are essential for our context. In the following, ordinary letters such as $x,u$ take values in 
$\R$ or $\Z$, while boldface letters such as $\x,\bu$ take 
values in $\R^d$ for $d\ge2$.  

Fix $d\ge2, N\ge1$ and define our parameter set as follows:
\begin{equation*}
S^d_{N}:=[-N,N]^d \cap \Z^d. 
\end{equation*}
Note that
\begin{equation*}
S^1_{N}:=\{-N,\dots,N\}.
\end{equation*}
Thus $S^d_N$ is a cube in $\Z^d$ centered at the origin.  

We denote by $\Delta=\Delta_{N,d,D}$  the discrete Neumann Laplacian on 
$S^d_{N}$ (see Section 1 of \cite{muel-neum-2021} for the precise 
definition). In the case $D=1$, since $\Delta$ is a self-adjoint operator on 
a finite-dimensional space, there exists a finite index set $\I=\I_{N,d}$ to 
be defined later, and an orthonormal basis of eigenfunctions 
$(\varphi_\bk)_{\bk\in\I}$ with corresponding eigenvalues 
$(\lambda_\bk)_{\bk\in\I}$.  We can assume without loss of generality that there is a 
distinguished index $\0$ such that $\varphi_\0$ is constant and that 
$\lambda_\0=0$.  

Throughout, we fix a parameter $\beta>0$, which has a physical 
interpretation as the inverse temperature.  
Let $(X^{(i)}_\bk)_{\bk\in\I\setminus\{\0\},i=1,\dots,D}$ be a collection of 
i.i.d. random variables defined on a probability space 
$(\Omega,\mathcal{F},\mathbb{P})$ such that
\begin{equation*}
X^{(i)}_\bk \sim N(0,(2\beta\lambda_\bk)^{-1}).
\end{equation*}
For each $i=1,\dots,D$ define
\begin{equation} \label{u-eigen}
u^{(i)}=\sum_{\bk\in\I\setminus\0}X^{(i)}_\bk\varphi_\bk,  \\
\end{equation}
and define DGFF as  
\begin{equation} \label{DGFF-def}
\bu=(u^{(1)},\dots,u^{(D)}).
\end{equation}
As explained in \cite{muel-neum-2021}, this corresponds to the Gibbs 
measure with Hamiltonian $H(\bu)=\sum_{\x\sim \y}|\bu(\x)-\bu(\y)|^2$, where 
$\x\sim \y$ means that $\x,\y$ should be nearest neighbors on $S_N^d$. In other 
words, the energy $H(\bu)$ depends on the stretching $|\bu(\x)-\bu(\y)|$.

We recall the definition the local time of the field $\bu$ at level 
$\z \in \mathbb R^{D}$ as 
\begin{equation} \label{l-time}
\begin{split} 
\ell_{N}(\z) &= \#\{\x \in S^d_{N} 
       : \bu(\x) \in [\z-\mathbf{1/2}, \z+\mathbf{1/2}] \} \\
&= \sum_{\x \in S_N^d} \mathbf{1}_{\{\bu(\x) \in [\z-\mathbf{1/2}, \z+\mathbf{1/2}] \}}, 
\end{split} 
\end{equation} 
where $\mathbf{\frac{1}{2}}=(1/2,\dots,1/2)\in\mathbb{R}^D$, 
and $[\x,\y]=\prod_i[x_i,y_i]$ for $\x,\y\in\mathbb{R}^D$.

Now we define a weakly self-avoiding Gaussian free field. 
Throughout, we fix a parameter $\gamma>0$.  
If $P_{N,d,D,\beta}$ denotes the 
original probability measure of $(\bu(\x))_{\x\in S_N^d}$, 
we define the probability $Q_{N,d,D,\beta,\gamma}$ as follows.  
For ease of notation, we write $E$ for the expectation with respect to 
$P_{N,d,D,\beta}$.
Let
\begin{equation} \label{z-func} 
\begin{aligned}
\mathcal{E}_{N,d,D,\gamma}
&=\exp\left(-\gamma \int_{\mathbb R^D}\ell_N(\y)^2d\y \right),\\
Z_{N,d,D,\beta,\gamma}&=E[\mathcal{E}_{N,d,D,\gamma}]
 =E^{P_{N,d,D,\beta}}[\mathcal{E}_{N,d,D,\gamma}].
\end{aligned}
\end{equation} 
Then we define for any set $A\in\mathcal{F}$,
\begin{equation} \label{eq:def-Q}
Q_{N,d,D,\beta,\gamma}(A)=\frac{1}{Z_{N,d,D,\beta,\gamma}}E
 \big[\mathcal{E}_{N,d,D,\gamma}\mathbf{1}_A\big].
\end{equation}
For ease of notation, we will often drop the subscripts except for 
$N$ and write
\begin{equation*}
P_N=P_{N,d,D,\beta} , \quad Q_N=Q_{N,d,D,\beta,\gamma}  , \quad 
\mathcal{E}_N=\mathcal{E}_{N,d,D,\gamma}. 
\end{equation*}
For $Z_{N,d,D,\beta,\gamma}$ in \eqref{z-func} we often write,  
$$
Z_{N,d,D}=Z_{N,d,D,\beta,\gamma}.
$$

Finally, we define the effective radius of the field $\bu$ as   
\begin{equation*}
R_{N,d,D}=  \max_{\w, \z\in S^d_N}  \|\bu(\z)-\bu(\w)\|, 
\end{equation*}
where $\|\cdot\|$ denotes the Euclidian norm.

\subsection{Statement of the main result}

Note that in our main theorem below, we assume that $D \leq d$.   
 
\begin{theorem}
\label{th:main}
Let $\bu$ be the weakly self-avoiding DGFF on $S_N^d$ taking values in 
$\mathbb{R}^D$.  There are constants $\varepsilon_0,K_0>0$ not depending on 
$\beta,\gamma,N$ such that  
\begin{itemize} 
\item[\textbf{(i)}] For $d=2$ and $D=1$, 
\begin{align*}
\lim_{N\to\infty} Q_{N}\Big[\varepsilon_0  \frac{\gamma}{\beta+\gamma}   N^{4/3}(\log N)&^{-2/3}
\leq R_{N,d,D} \\
 &\leq K_{0} \Big(\frac{\beta+\gamma}{\beta}\Big)^{1/2}  N^{4/3}(\log N)^{4/3}\Big]  =1.
\end{align*}
\item[\textbf{(ii)}] For $d\geq 3$ and $1\leq D \leq d$, 
\begin{align*}
\lim_{N\to\infty} Q_{N}\Big[\varepsilon_0  \Big(\frac{\gamma}{\beta+\gamma}\Big)^{1/D}N^{\frac{1}{D} \left(d-\frac{2(d-D)}{D+2} \right)} &
 \leq R_{N,d,D} \\
& \leq K_{0}\Big(\frac{\beta+\gamma}{\beta}\Big)^{1/2}N^{\frac{d}{2}+\frac{d-D}{D+2}}  \Big] =1.
\end{align*}
\end{itemize} 
\end{theorem}
 \begin{remark}
We compare the result of Theorem \ref{th:main}(i) to the result of Theorem 
1.1(i) in \cite{muel-neum-2021}, where it was proved that  
$R_{N,2,2} \approx N$. Note that by reducing the dimension of the range by 
1, the manifold stretches, with a substantially larger radius of
$R_{N,2,1} \approx N^{4/3}$. By comparing  Theorem \ref{th:main}(ii) and 
Theorem 1.1(ii) in \cite{muel-neum-2021} we observe a similar phenomenon for 
$d\geq3$, as the lower bound on the radius increases from 
$R_{N,d,d} \gtrsim N$ to 
$R_{N,d,D} \gtrsim N^{ \frac{1}{D} \left(d-\frac{2(d-D)}{D+2} \right)}$ for 
$D<d$. This additional stretching for manifolds of lower dimensional range 
can be predicted by considering the local time expression in \eqref{l-time}. 
Indeed for a fixed value of $d$ and when the range has dimension $D<d$, 
$\ell_{N}(\z)$ counts the same number of vertices $(2N+1)^d$, but if the 
radius $R_N$ remains the same, there is less space to fit these vertices.  
Hence we would expect $\mathcal{E}_N$ to be larger. This in turn gives the 
repelling term in \eqref{z-func} a stronger influence on the configuration 
of the manifold (see e.g. \eqref{l-lower} in Section \ref{sec-small}).   
\end{remark} 

 \begin{remark} 
 Theorem \ref{th:main} verifies the conjecture by Kantor, Kardar and Nelson in \cite{Kardar_87} for the case where $d=2$ and $D=1$. Although in the model that was presented in \cite{Kardar_87} the DGFF is defined on the triangular lattice, the heuristics that yields their result is based on Flory's argument which also applies for the rectangular lattice.   
 \end{remark} 
 
\subsection{Outline of the proof}
We describe the outline for the case $d=2$, $D=1$, as the proof for $d\geq 3$, $D <d$ follows similar lines. Define the following events.  
\begin{equation} \label{events}
\begin{aligned} 
A^{(1)}_{N,d,D}&=\left\{R_{N,d,D}>K_{0} \Big(\frac{\beta+\gamma}{\beta}\Big)^{1/2}  N^{4/3}(\log N)^{4/3} \right\}, \\
A^{(2)}_{N,d,D}&=\left \{R_{N,d,D}<\varepsilon_0  \frac{\gamma}{\beta+\gamma}   N^{4/3}(\log N)^{-2/3} \right \}. 
\end{aligned}
\end{equation}
It suffices to show that for $i=1,2$ we have
\begin{equation*}
\lim_{N\to\infty}Q_{N}\big(A^{(i)}_{N,d,D}\big)=0 .
\end{equation*}

From \eqref{eq:def-Q} we see that it is enough to find:
\begin{enumerate}
\item a lower bound on $Z_{N,d,D}$, derived in Section \ref{sec-part},
\item and an upper bound on 
$E \big[\mathcal{E}_{N,d,D}\mathbf{1}_{A^{(i)}_{N,d,D}}\big]$ 
for $i=1,2$, obtained in Sections \ref{sec-large-t} and \ref{sec-small}, respectively. 
\end{enumerate}
Finally, the upper bounds divided by the lower bound should vanish 
as $N\to\infty$.  

\section{ Lower Bound on the Partition Function}  \label{sec-part} 
In this section we derive the following lower bound on $ Z_{N,d,D}$.  
\begin{prop} \label{prop-z} 
Let $\beta>0$. Then there exists a constant $C>0$ not depending on $N$, $\beta$ and $\gamma$ such that  
\begin{itemize} 
\item[\textbf{(i)}] for $d=2$ and $D=1$, 
\begin{equation*} \label{z-bnd-fin} 
\log  Z_{N,d,D} \geq  -C(\beta+\gamma) N^{8/3}(\log N)^{2/3}. 
\end{equation*}
\item [\textbf{(ii)}] for $d\geq 3$ and $D\leq d$, 
 \begin{equation*} \label{z-bnd-fin-3} 
\log  Z_{N,d,D} \geq  -C(\beta+\gamma) N^{d+\frac{2(d-D)}{d-D+2}}. 
\end{equation*} 
\end{itemize} 
\end{prop} 
In order to prove Proposition \ref{prop-z} we will introduce some additional definitions and auxiliary lemmas.

%

\subsection{The orthonormal function basis}  \label{sec-basis} We first 
recall the orthonormal basis $\{\varphi_{\bk}\}$ in \eqref{u-eigen} of 
eigenfunctions of $\Delta_{N,d,1}$ on 
$[-N, N]^d \cap \mathbb{Z}^d$ taking values in $\mathbb{R}$.  
Each basis function $\varphi_{\bk}$ can be represented as a product of $d$ 
functions $\phi_{j}: \mathbb R^d \rightarrow \mathbb R$, as follows 
\begin{equation} \label{eigen}
\varphi_{\bk} (\x) = \phi_{k_1}(x_1)\dots\phi_{k_d}(x_d), 
\end{equation} 
where $\x=(x_1,\dots,x_d)$ and $\bk=(k_1,\dots,k_d)$, $-N \leq k_i \leq N$, 
and $1\leq i\leq d$.  Here $\{\phi_j\}_{j=-N}^N$ is an orthonormal basis of 
eigenfunctions of $\Delta_{N,1,1}$, the Laplacian with Neumann boundary 
conditions on $S_N^1=\{-N,\dots,N\}$. 
Note that if $\lambda_{\bk}$ is the eigenvalue of $\varphi_{\bk}$ and 
$\lambda_k$ is the eigenvalue corresponding to $\phi_k$, then
satisfies
\begin{equation} \label{eigen-rel} 
 \lambda_{\bk} = \sum_{i=1}^d \lambda_{k_i}. 
\end{equation} 
We can explicitly compute these eigenfunctions and eigenvalues (see Section \ref{sec-basis} of \cite{muel-neum-2021}).  
 
Our basis comprises all such combinations as in \eqref{eigen}, excluding the constant eigenfunction 
\begin{equation*}  
 \varphi_{(0,\dots,0)} (\x) = \phi_0(x_1)\dots\phi_{0}(x_d). 
\end{equation*} 
We denote by $N(d)$ the number of function in our basis,
\begin{equation} \label{n-d} 
 N(d) = |S_N^d|-1= (2N+1)^d-1. 
\end{equation}

\subsection{Incorporating drift} 
 
 
Next we incorporate a linear drift into each of the components $u^{(i)}$ of 
$\bu$, calling the resulting component $u^{(i)}_a$.  
This drift should stretch out the values of $\bu(\cdot)$, so that $\bu$ 
looks more like what we believe it would be under $Q_N$.  Since our goal is to use 
Jensen's inequality, this drift should make the inequality more exact.  
Let
\begin{equation*} 
u^{(i)}_a(\x)
  =\sum_{k=1}^{N(d)}X^{(i)}_\bk\varphi_\bk(\x) + ax_i, \quad i=1,\dots, D,
\end{equation*}
where $a>0$ is a constant to be determined later. 

Using \eqref{u-eigen} and \eqref{sec-basis}, we get
\begin{equation} \label{u-drf}
u_a^{(i)}(\x) =\sum_{(k_1,\dots,k_d) \in S_N^d \setminus \{\mathbf{0}\}  } X^{(i)}_{k_1,\dots,k_d}\phi_{k_1}(x_1)\dots.\phi_{k_d}(x_d) +ax_i,
\end{equation}
where $\mathbf{0} = (0,\dots,0)\in\mathbb{Z}^d$. 

Regarding $x_i$ as a function on $S_N^1=\{-N,\dots,N\}$ and expanding it 
in terms of our eigenfunctions, we see that there are coefficients 
$\alpha_j^{(i)}$ such that 
\begin{equation} \label{u-drf2}
x_i=  (\phi_{0})^{d-1} \sum_{j\in S_N^1\setminus \{0\}}  \phi_{j}(x_i) \alpha_{j}^{(i)}, 
\end{equation}
where we have included $(\phi_0)^{(1-d)}$ for convenience in later 
calculations.  Recall that 
\begin{equation*}
\phi_0=\phi_{0}(x)=(2N+1)^{-1/2}. 
\end{equation*}
In \eqref{u-drf2}, we do not include $j=0$ because $x_i$ is orthogonal to 
the constant function $\phi_0$.  

Next we find $\alpha_j^{(i)}$ in \eqref{u-drf2}. Since 
$\alpha_j^{(i)}$ are used to expand the function $f(x)=x$ for each 
coordinate $i$, we can omit the superscript $i$ and write just $\alpha_j$ in 
what follows.  Since $\{ \phi_j \}_{j=-N}^N$ forms an orthonormal basis, we 
get  
\begin{equation}  \label{f-coef} 
\alpha_{j} =\phi_0^{(1-d)} \sum_{n=-N}^{N}n \phi_j(n), \quad j\not =0, \quad \textrm{and } \quad \alpha_{0}=0. 
\end{equation} 

From \eqref{u-drf} and \eqref{u-drf2} we have  
 \begin{equation*}
 \begin{aligned} 
u_a^{(i)}(\x)& =\sum_{(k_1,\dots,k_d) \in S_N^d \setminus \{\textbf{0}\}   } X^{(i)}_{k_1,\dots,k_d}\phi_{k_1}(x_1)\dots.\phi_{k_d}(x_d) \\
&\quad +a\phi_{0}^{d-1} \sum_{j \in S_N^1 \setminus \{0\}}  \phi_{j}(x_i) \alpha_{j}. 
\end{aligned} 
\end{equation*}
We can represent $u_a^{(i)}$ as follows: 
\begin{equation} \label{u-drift} 
 \begin{aligned} 
u_a^{(i)}(x)& =\sum_{(k_1,\dots,k_d) \in S_N^d \setminus \{j\e_i: \, j\in S_N^1\}} X^{(i)}_{k_1,\dots,k_d}\phi_{k_1}(x_1)\cdots\phi_{k_d}(x_d) \\
&\quad +\phi_{0}^{d-1} \sum_{j \in S_N^1\setminus \{0\}}  (X_{j\e_i}^{(i)} +a \alpha_{j}^{(i)}) \phi_{j}(x_i),
\end{aligned} 
\end{equation}
where $\{\e_i\}_{i=1}^d$ is the standard basis of $\mathbb{R}^d$. 

For $i=1,\dots,D$ let $\x^{(i)} \in \mathbf{V}$ and define 

\begin{equation*}
F(\bold x^{(1)},\dots,\bold x^{(D)}) = \sum_{i=1}^{D}\sum_{(k_1,\dots,k_d)\in S_N^d\setminus\{\mathbf{0}\}}\frac{(x^{(i)}_{k_1,\dots,k_d})^2}{2(2\beta\lambda_{k_1,\dots,k_d})^{-1}}.
\end{equation*}
We rewrite $Z_{N,d,D}$ in \eqref{z-func} as follows.  We should emphasize that 
the local time $\ell_N$ is random and hence a function of the random 
variables $(X^{(i)}_{\bk})$, so we can write
\begin{equation*}
\ell_N(\y) 
= \ell_N\left((X_{\bk}^{(i)}),\y\right),
\end{equation*}
where for readability we have omitted the specification that $i=1,\dots,D$ 
and $\bk\in S_N^d\setminus\{\mathbf{0}\}$. Then we have
\begin{equation} \label{z-hat-mod} 
\begin{aligned}
Z_{N,d,D}
&=  \int_{\mathbb{R}^{D\cdot N(d)}} \exp\left(-F(\bold x^{(1)},\dots,\bold x^{(D)})  -\gamma \int_{\mathbb{R}^D} \ell_N^{2}\big((\x^{(i)}),\y\big)d\y\right) \\
&\hspace{2cm} \times \prod_{i=1}^D\prod_{(k_1,\dots,k_d)\in S_N^{d}\setminus\{\textbf{0}\}}dx_{k_1,\dots,k_d}^{(i)}  . 
\end{aligned}
\end{equation}  
In order to find the Radon-Nikodym derivative corresponding to the drift in \eqref{u-drift} we note that, 
\begin{equation} \label{sums-sums} 
\begin{aligned} 
&\sum_{i=1}^{D}\sum_{(k_1,\dots,k_d)\in S_N^{d}\setminus\{\textbf{0}\}}\frac{(x^{(i)}_{k_1,\dots,k_d})^2}{2(2\beta\lambda_{k_1,\dots,k_d})^{-1}}  \\ 
&= \sum_{i=1}^{D} \bigg( \sum_{ (k_1,\dots,k_d) \in S_N^d \setminus \{j\textbf{e}_{i}: \, j\in S_N^1 \}}\frac{(x^{(i)}_{k_1,\dots,k_d})^2}{2(2\beta\lambda_{k_1,\dots,k_d})^{-1}}  \\
&\qquad  + \sum_{j\in S_N^1 \setminus \{0\}}\frac{(x^{(i)}_{ j\e_{i}} +a\alpha_{j})^2}{2(2\beta\lambda_{j\e_i})^{-1}} -\sum_{j\in S_N^1 \setminus \{0\}}  \frac{2a\alpha_{j}x^{(i)}_{j\e_i}+(a\alpha_{j})^{2}}{2(2\beta\lambda_{j\e_i})^{-1}}  \bigg).\end{aligned} 
\end{equation}  

We therefore define $\hat P^{(a)}$ (resp. $\hat E^{(a)}$) to be the probability measure 
(expectation) under which $u$ is shifted as in \eqref{u-drift}. Then 
\eqref{z-hat-mod} and \eqref{sums-sums} imply
\begin{equation} \label{hat-P-meas} 
\frac{d\hat P^{(a)}}{dP}= \exp \bigg(-\sum_{i=1}^{D}\sum_{j\in  S_N^1 \setminus \{0\}}  \frac{2a\alpha_{j}x^{(i)}_{j\e_i}+(a\alpha_{j})^{2}}{2(2\beta\lambda_{ j\e_i})^{-1}} \bigg) .
\end{equation}
 
We can therefore rewrite  $Z_{N,d,D}$ in \eqref{z-hat-mod} as follows,  
\begin{equation} 
\begin{split}
Z_{N,d,D}
&= \hat E^{(a)}\Bigg[  \exp\bigg( \sum_{i=1}^{D}\sum_{j\in S_N^1 \setminus \{0\}}  \frac{2a\alpha_{j}X^{(i)}_{j\e_i}+(a\alpha_{j})^{2}}{2(2\beta\lambda_{j\e_i})^{-1}} \\
& \hspace{4cm} -\gamma \int_{\mathbb{R}^D} \ell_N^{2}(\y)d\y\bigg) \Bigg]. 
\end{split}
\end{equation}   
 
We define  
\begin{equation} \label{y-rv} 
Y_{j\e_i}^{(i)} =    \frac{2a\alpha_{j}X^{(i)}_{j\e_i}+(a\alpha_{j})^{2}}{2(2\beta\lambda_{j\e_i})^{-1}}.
\end{equation}

Using Jensen's inequality, we get that 
\begin{equation} \label{log-z} 
\begin{aligned}
\log   Z_{N,d,D}  
&\geq \hat E^{(a)}\left[-\gamma \int_{\mathbb{R}^D} \ell_N^{2}(\y)d\y\right] 
  - \hat E^{(a)}\left[- \sum_{i=1}^D\sum_{j\in S_N^1\setminus 
\{0\}}Y_{j\e_i}^{(i)}    \right]  \\
&=:-(I_{1,d,D} +I_{2,d,D}).
\end{aligned}
\end{equation}
The following proposition gives some essential bounds on $I_{i,d,D}$, $i=1,2$. 
\begin{prop} \label{prop-bnd-I} 
Let $\beta,\gamma >0$. Then there exists a constant $C>0$ not depending on 
$N,\beta,\gamma$ such that   

\begin{itemize} 
\item [\textbf{(i)}]  for $d=2$ and $D=1$,
\begin{equation*}
I_{1,2,1}   \leq  C\gamma  N^{2}  \big( (\beta^{-1/2} a^{-1} N \log N) \vee 1 \big),
\end{equation*}
\item [\textbf{(ii)}] for  $d\geq 3$ and $D\leq d$,
\begin{equation*}
I_{1,d,D}  \leq  C \gamma  N^d\big( 1 \vee ( \beta^{-D/2}  N^{d-D}a^{-D}) \big). 
\end{equation*}
\item  [\textbf{(iii)}]  for any $d\geq 2$, 
\begin{equation*}
I_{2,d,D}  \leq C\beta  a^2 N^d.
\end{equation*}
\end{itemize} 
\end{prop} 
The proof of Proposition \eqref{prop-bnd-I}(i) and (ii) is postponed to Section \ref{sec-pf-z-1}. The proof of Proposition \eqref{prop-bnd-I}(iii) is given in Section \ref{sec-pf-z-2}. 
\subsection{Proof of Proposition \ref{prop-z}}  
\begin{proof} [Proof of Proposition \ref{prop-z}] 
From \eqref{log-z} and Proposition \ref{prop-bnd-I}(i) and (iii) it follows that for $d=2$ and $D=1$,  
\begin{equation} \label{z-lb} 
\begin{aligned} 
\log \hat Z_{N,2,1} &\geq -( I_{1,2,1}+I_{2,2,1}) \\ 
&\geq  -C \left[\gamma  N^{2}  \big( (\beta^{-1/2} a^{-1} N \log N) \vee 1 \big)  +\beta N^2  a^2\right].
\end{aligned} 
\end{equation} 
Taking $a= \beta^{-1/2}(N\log N)^{1/3}$ in \eqref{z-lb} gives, 
\begin{equation*}
 \log \hat Z_{N,2,1} \geq  -C(\beta + \gamma) N^{8/3}(\log N)^{2/3}. 
\end{equation*}

The proof for $d \geq 3$ and $1\leq D < d$ follows the same lines with the only modification that we are using Proposition \ref{prop-bnd-I}(ii) and choosing $a=\beta^{-1/2} N^{\frac{d-D}{D+2}}$ to get 
\begin{equation*}
 \log \hat Z_{N,d,D} \geq  -C(\beta + \gamma) N^{d+\frac{2(d-D)}{D+2}}.
\end{equation*} 
\end{proof}

\section{Proof of Proposition \eqref{prop-bnd-I}(i) and (ii)}  \label{sec-pf-z-1} 
\begin{proof} [Proof of Proposition \ref{prop-bnd-I}(i) and (ii)] 
 We can write 
\begin{equation} \label{i-1-eq} 
\begin{aligned} 
\tilde I_{1,d,D}&:=\hat{E}^{(a)}\left[ \int_{\mathbb R^D} \ell_N(\y)^2d\y \right] \\
&= \hat{E}^{(a)}\left[ \int_{\mathbb{R}^D} \Big( \sum_{\z\in S^d_N}\mathbf{1}_{[\y-\mathbf{1/2},\y+\mathbf{1/2}]}(\bu(\z)) \Big)^2d\y  \right]  \\ 
&=  \sum_{\z\in S^d_N} \hat{E}^{(a)}  \left[\int_{\mathbb R^D} \mathbf{1}_{[\y-\mathbf{1/2},\y+\mathbf{1/2}]}(\bu(\z))d\y \right] \\ 
&\quad + \sum_{\z,\w\in S^d_N, \, \z\not =\w}\hat{E}^{(a)}  \left[\int_{\mathbb R^D} \mathbf{1}_{\bu(\z) ,\bu(\w) \in [\y-\mathbf{1/2},\y+\mathbf{1/2}]} d\y \right]  \\
&=  (2N+1)^d 
 +  \sum_{\z,\w\in S^d_N,\,\z\not=\w}\int_{\|\y\|\le1}
          \hat p_{\z,\w}(\y)d\y,
\end{aligned} 
\end{equation}
where $ \hat p_{\z,\w}$ is the density of $\bu(\z)-\bu(\w)$ under $\hat P^{(a)}$. 
\\

We recall Proposition 3.1 from \cite{muel-neum-2021}.  
\begin{prop} \label{lem-cov} 
There exist constants $C_1,C_2>0$ such that, 
\begin{itemize} 
\item[\textbf{(i)}] for $d=2$, for all $\w,\z\in S_N^2$ with 
$\w\ne\z$, and for $i=1,2$ we have
\begin{equation*} 
C_1 \beta^{-1} \leq \textrm{Var}\left(u^{(i)}(\z)-u^{(i)}(\w)\right) \leq C_2 \beta^{-1}  (\log N)^2,
\end{equation*} 
\item[\textbf{(ii)}] for $d \geq 3$, for all $\w,\z\in S_N^d$ with 
$\w\ne\z$, and for $i=1,\dots,D$ we have
\begin{equation*} 
C_1 \beta^{-1}  \leq \textrm{Var}\left(u^{(i)}(\z) - u^{(i)}(\w)\right) \leq C_2 \beta^{-1 }  .
\end{equation*} 
\end{itemize} 
\end{prop} 
Note that from \eqref{u-drf} we have  
\begin{equation*}
\hat E^{(a)}[ u^{(i)}(\z) - u^{(i)}(\w) ]   = a(z_i-w_i), \quad 
\text{ for }i=1,\dots,D.   
\end{equation*}
Since $(u^{(i)})_{i=1,\dots,D}$ are independent, we have for any $\bold y \in \mathbb{R}^D$
\begin{equation} \label{prod}  
 \hat p_{\z,\w}(\bold y) := \prod_{i=1}^D\hat p^{(i)}_{\z,\w}( y_i)
\end{equation} 
and therefore
\begin{equation} \label{prod-1}  
 \int_{\|\y\|\le1}\hat{p}_{\z,\w}(\y)d\y
\le  \prod_{i=1}^D\int_{-1}^{1}\hat{p}^{(i)}_{\z,\w}(y_i)dy_i, 
\end{equation} 
where $\hat{p}^{(i)}_{\z,\w}$ is the density of $u^{(i)}(\z)-u^{(i)}(\w)$ under $\hat P^{(a)}$. 

We distinguish between the following two cases.  

 \textbf{Case 1:} $d \geq 3$ and $D \leq d$.  From Proposition \ref{lem-cov}(ii) we have 
\begin{equation} \label{p-bnd-3}  
 \hat p^{(i)}_{\z,\w}( y_i) \leq  \frac{1}{\sqrt{2\pi C_1\beta^{-1}}} \exp\Big( -\frac{a^2(z_i-w_i-y_i)^2}{2C_2 \beta^{-1}} \Big), \quad \textrm{i=1,...,D}. 
 \end{equation}
From \eqref{prod-1} and \eqref{p-bnd-3} we therefore get  \begin{align*}   \label{sum-p-bnd}
\sum_{\z,\w\in S_N^d, \, \z\not =\w}  \int_{\| \y\| \leq1}  \hat p_{\z,\w}(    \y) 
& \leq C\sum_{\z,\w\in S_N^d, \, \z\not =\w}  \int_{ \y \in [-1,1]^D}      \frac{1}{(2\pi C_1\beta^{-1})^{d/2}} \\
&\qquad \times\exp\Big( -\frac{a^2 \sum_{i=1}^D(z_i-w_i-y_i)^2}{2C_2\beta^{-1}} \Big)d\y \\
& \leq  \int_{\y \in [-1,1]} J(\y) d\y. 
\end{align*}
where 
$$
J(\y):=\sum_{\z,\w\in S_N^d}   \frac{  1}{2\pi C_1\beta^{-1}}\exp\Big( -\frac{a^2 \sum_{i=1}^D(z_i-w_i-y_i)^2}{2C_2\beta^{-1}} \Big).
$$
The following lemma follows immediately from the proof of Lemma 3.2 from \cite{muel-neum-2021}. 
\begin{lemma}\label{lem-cont-sum} 
Let $\kappa>0$. Then for all $y\in [-1,1]$ and $w \in S^1_N$ we have 
\begin{align*} 
 \sum_{z\in S^1_N}    \exp\Big( -\kappa (z-w -y)^2 \Big) 
&\leq  \sum_{z= w-1}^{w+1} \exp\Big( -\kappa (z -w - y )^2 \Big)   \\
&\quad +  \int_{-\infty}^{\infty} \exp\Big( -\kappa (z-w - y)^2 \Big)dz. 
\end{align*} 
\end{lemma}  

Using Lemma \ref{lem-cont-sum} and integrating over the Gaussian density gives,    
\begin{equation} \label{j-bnd} 
\begin{aligned} 
J(\y)   
&= \frac{ (2N+1)^{2(d-D)} }{2\pi C_1\beta^{-1} }
   \prod_{i=1}^{D}\left(\sum_{z_i\in S_N^1}\sum_{w_i\in S_N^1}   \exp\Big( -\frac{a^2(z_i-w_i - y_i)^2}{2C_2\beta^{-1} } \Big)\right)\\
&\leq C N^{2(d-D)}  
   \prod_{i=1}^{D}\sum_{w_i\in S_N^1}\Bigg( \sum_{z_i= w_i-1}^{w_i+1} \exp\Big( -\frac{a^2(z_i-w_i - y_i)^2}{2C_2\beta^{-1} } \Big) \\
   &\quad + \beta^{-1/2}a^{-1}\frac{1}{\sqrt{ 2\pi  C_2\beta^{-1} a^{-2}  }} 
 \int_{-\infty}^\infty\exp\Big( -\frac{(z_i-w_i-y_i)^2}{2  C_2\beta^{-1} a^{-2}} \Big) dz_i \Bigg)  \\
&\leq C  N^{2(d-D)} \prod_{i=1}^{D} \sum_{w_i \in S_N^1} \left( \sum_{z_i= w_i-1}^{w_i +1} \exp\Big( -\frac{a^2(z_i-w_i - y_i)^2}{2C_2\beta^{-1} }\Big) +\beta^{-1/2} a^{-1}\right). 
\end{aligned}
\end{equation}
It follows that
\begin{equation} \label{n-1} 
\begin{aligned} 
& \int_{\y \in [-1,1]} J(\y) d\y \\
&\leq C  N^{2(d-D)}  \left( \int_{-1}^{1}\sum_{w \in S_N^1} \left( \sum_{z= w-1}^{w+1} \exp\Big( -\frac{a^2(z-w - y)^2}{2C_2\beta^{-1} }\Big) +\beta^{-1/2}a^{-1}\right) dy \right)^D \\
&\leq C  N^{2(d-D)} N^D\left(  \int_{-1}^{1} \left( \sum_{k= -1}^{1} \exp\Big( -\frac{a^2(k - y)^2}{2C_2\beta^{-1} }\Big) +\beta^{-1/2} a^{-1}\right) dy \right)^D.
\end{aligned}
\end{equation}
Using again integration over the Gaussian density gives for any $M>0$ and $k=-1,0,1$,  
$$
\int_{-1}^1 e^{-M (k+x)^2} dx   \leq \int_{-\infty}^\infty e^{-M (k+x)^2} dx \leq C M^{-1/2}.
$$
Plugging in these bounds to \eqref{n-1} gives 
 $$
 \int_{\y \in [-1,1]} J(\y) d\y  \leq C  N^{2(d-D)} N^D(\beta^{-1/2} a^{-1} )^D . 
$$
Together with \eqref{i-1-eq} this leads to, 
\begin{equation*}
\tilde I_{1}  \leq \tilde C N^d\big( 1 \vee ( \beta^{-D/2}  N^{d-D}a^{-D}) \big). 
\end{equation*}

 \textbf{Case 2:} $d=2$ and $D=1$. Then from Proposition \ref{lem-cov}(i) we have 
\begin{equation} \label{p-bnd}  
 \hat p^{(1)}_{\z,\w}( y) \leq  \frac{1}{\sqrt{2\pi C_1 \beta^{-1}}} \exp\Big( -\frac{a^2(y-(z_1-w_1))^2}{2\beta^{-1} C_2(\log N)^2} \Big).
 \end{equation}
Then following similar steps as in Case 1 we arrive to  
\begin{equation} \label{i-1} 
\begin{aligned} 
\tilde I_{1,2,1} 
&\leq  (2N+1)^2 + C \beta^{-1/2} a^{-1 }N^3\log N \\ 
&\leq \tilde C \big( (\beta^{-1/2} a^{-1} N \log N) \vee 1 \big)  N^{2}. 
\end{aligned}
\end{equation}
Since from \eqref{log-z} and \eqref{i-1-eq} we have that 
\begin{equation*}
I_{1,d,D} = \gamma  \tilde I_{1,d,D},  
\end{equation*}
this completes the proof of Proposition \ref{prop-bnd-I} parts (i) and (ii). 
\end{proof}

\section{Proof of Proposition \eqref{prop-bnd-I} (iii) } \label{sec-pf-z-2}
\begin{proof}[ Proof of Proposition \eqref{prop-bnd-I}(iii)] 
Recall that $I_{2,d,D}$ was defined in \eqref{log-z}, 
\begin{equation} \label{l2-dec} 
\begin{aligned} 
I_{2,d,D} &= -\sum_{i=1}^D\sum_{j \in S_N^1\setminus \{0\}} \hat E^{(a)}\left[ Y_{j\e_i}^{(i)}    \right], 
\end{aligned}
\end{equation}
where $Y_{ j\e_i}^{(i)}$ was defined in \eqref{y-rv}. Recall also that $N(d)$ was 
defined in \eqref{n-d}. Since the expectation on the right-hand side of  
\eqref{l2-dec} is taken over a Gaussian measure, we define the following 
normalizing constant 
 \begin{equation} \label{z-n} 
\begin{aligned} 
C_{N,\beta,d,D} &= \int_{\mathbb{R}^{D \cdot N(d)} }
  \exp\left(-\sum_{i=1}^{D}\sum_{k=1}^{N(d)}
   \frac{(x^{(i)}_k)^2}{2(2\beta\lambda_k)^{-1}}\right) 
   \prod_{i=1}^{D} \prod_{k=1}^{N(d)}dx^{(i)}_k\\
&=  \frac{1}{(2\beta)^{DN(d)/2}}\prod_{k=1}^{N(d)}\frac{1}{\lambda_k^{D/2}}. 
\end{aligned}  
\end{equation}
We refer to Section 1.2 in \cite{muel-neum-2021} for additional in formation about the setup of the GFF measure.

We further introduce the following notation: 
\begin{equation*}
z_{k_1,\dots,k_d}^{(i)} =\frac{(x^{(i)}_{k_1,\dots,k_d})^2}{2(2\beta\lambda_{k_1,\dots,k_d})^{-1}}  , \quad w^{(i)}_{j\e_i} = \frac{(x^{(i)}_{j\e_i} +a\alpha_{j})^2}{2(2\beta\lambda_{j\e_i})^{-1}}, 
\end{equation*}
and 
\begin{equation*}
y_{j\e_i}^{(i)} =    \frac{2a\alpha_{j}x^{(i)}_{j\e_i}+(a\alpha_{j})^{2}}{2(2\beta\lambda_{j\e_i})^{-1}}.
\end{equation*}
Then from \eqref{sums-sums} and \eqref{hat-P-meas} we have 
\begin{equation}  \label{y-calc} 
\begin{aligned}
& \hat E^{(a)}\left[  Y_{j\e_i}^{(i)} \right]  \\
&= \frac{1}{C_{N,\beta,d,D} } \int y_{j\e_i}^{(i)}  \exp\bigg( -\sum_{i=1}^D \Big( \sum_{ (k_1,\dots,k_d) \in S_N^d \setminus \{j\e_i: \, j\in S_N^1\}}z_{k_1,\dots,k_d}^{(i)}  \\
& \qquad    + \sum_{l\in S_N^1 \setminus \{0\}}w^{(i)}_{l\e_i} \Big)\bigg)\prod_{i=1}^{D} \prod_{l=1}^{N(d)}dx^{(i)}_l,
\end{aligned}
\end{equation} 
where $C_{N,\beta,d,D}$ was defined in \eqref{z-n}.  

Since the expected value in \eqref{y-calc} is symmetric with respect to $i$, we can use $i=1$ in what follows in order to ease the notation. We therefore consider
\begin{equation}  \label{y-1-exp} 
\begin{split}
& \hat{E}^{(a)}\left[  Y_{j\e_1}^{(1)} \right]
  = \frac{1}{C_{N,\beta,d,D} } \int y_{j\e_1}^{(1)} \exp(-w^{(1)} _{j\e_1} )   \\
&\quad  \times \int \exp\bigg(- \Big(\sum_{i=1}^D  \sum_{ (k_1,\dots,k_d) \in S_N^d \setminus \{l\e_i: \, l\in S_N^1 \}}z_{k_1,\dots,k_d}^{(i)}  \\
& \quad    + \sum_{i=2}^D\sum_{l\in \{-N,\dots,N\}\setminus \{0\}}w^{(i)}_{l\e_i} +\sum_{l\in S_N^1 \setminus \{0,j\}}w^{(1)}_{l\e_1}    \Big)\bigg)\prod_{i=1}^{D} \prod_{l=1}^{N(d)}dx^{(i)}_l.
\end{split}
\end{equation} 
We notice that we have three types of integrals above, which can be evaluated as follows. 
We have $D((2N+1 )^d - (2N+1))$ integrals of the form 
\begin{equation*} 
\begin{aligned} 
\int_{\mathbb{R}} \exp \big(-z_{k_1,\dots,k_d}^{(i)} \big) dx^{(i)}_{k_1,\dots,k_d}&=
\int_{\mathbb{R}} \exp\left(- \frac{(x^{(i)}_{k_1,\dots,k_d})^2}{2(2\beta\lambda_{k_1,\dots,k_d})^{-1}}  \right) dx^{(i)}_{k_1,\dots,k_d} \\
& =  \sqrt{2\pi}  (2\beta\lambda_{k_1,\dots,k_d})^{-1/2}. 
\end{aligned}
 \end{equation*} 
We have $2N(D-1)+2N-1$ integrals of the form 
\begin{equation*} 
\begin{aligned}  
\int_{\mathbb{R}} \exp\big(-w^{(i)}_{l\e_i} \big)dx^{(i)}_{l\e_i}  &=
\int_{\mathbb{R}} \exp\left(- \frac{(x^{(i)}_{\textbf 0_{i}(l)} +a\alpha_{l})^2}{2(2\beta\lambda_{ l\e_i})^{-1}} \right) dx^{(i)}_{l\e_i} \\
& =  \sqrt{2\pi}  (2\beta\lambda_{l\e_i})^{-1/2},
\end{aligned}
\end{equation*} 
and one integral as follows 
\begin{equation*} 
\begin{aligned}  
& \int y_{j\e_1} \exp(-w^{(1)} _{j\e_1} )dx_{j\e_1}^{(1)} \\ &=\int_{\mathbb{R}}  \frac{2a\alpha_{j}x^{(1)}_{j\e_1}+(a\alpha_{j})^{2}}{2(2\beta\lambda_{j\e_1})^{-1}} \exp\left(-\frac{(x^{(1)}_{j\e_1} +a\alpha_{j})^2}{2(2\beta\lambda_{j\e_1})^{-1}} \right)dx_{j\e_1}^{(1)} \\ 
&= - \sqrt{2\pi} \frac{(a\alpha_{j})^2}{2(2\beta\lambda_{j\e_1})^{-1/2}}.
\end{aligned}
\end{equation*} 
Plugging in all the above integrals into \eqref{y-1-exp} gives 
\begin{equation*} 
\begin{aligned} 
& \hat E^{(a)}\left[  Y_{j\e_1}^{(1)} \right]\\
 &= -  \frac{1}{C_{N,\beta,d,D}} \sqrt{2\pi} \frac{(a\alpha_{j})^2}{2(2\beta\lambda_{j\e_1})^{-1/2}} \\
 &\qquad \times  \prod_{i=1}^D \prod_{ (k_1,\dots,k_d) \in S_N^d \setminus \{l\e_i: \, l\in S_N^1 \}}\sqrt{2\pi}  (2\beta\lambda_{k_1,\dots,k_d})^{-1/2} \\
 &\qquad \times \prod_{i=2}^D \prod_{l\in  S_N^1 \setminus \{0\}}\sqrt{2\pi}  (2\beta\lambda_{{ l\e_i}})^{-1/2} \prod_{l\in S_N^1 \setminus \{0,k\}} \sqrt{2\pi}  (2\beta\lambda_{ l\e_1})^{-1/2} \\
 &=    -  \frac{1}{C_{N,\beta,d,D}} \frac{(2\pi)^{((2N+1)^{d}-1)D/2}(a\alpha_{j})^2\lambda_{ j\e_1}}{2(2\beta)^{((2N+1)^{d}-1)D/2-1}}  \prod_{i=1}^D  \prod_{(k_1,\dots,k_d) \in S_N^d\setminus\{\textbf 0\}}\frac{1}{ \lambda_{k_1,\dots,k_d}^{1/2}}.
\end{aligned}
\end{equation*}
 
Together with \eqref{z-n} we get 
\begin{equation} \label{ir1} 
\begin{aligned}  
  \hat E^{(a)}\left[  Y_{j\e_1}^{(1)} \right] &=- \beta(a\alpha_{j})^2\lambda_{j\e_1}. 
\end{aligned}
\end{equation}

 Plugging  \eqref{ir1} into \eqref{l2-dec} gives 
\begin{equation} \label{i2-fin} 
\begin{aligned} 
I_{2,d,D}&= \beta \sum_{i=1}^D \sum_{j \in S_N^1 \setminus \{0\}}(a\alpha_{j})^2\lambda_{j\e_i}  \\
  &=D\beta a^2  \sum_{j \in S_N^1 \setminus \{0\}}\alpha_{j}^2\lambda_{j\e_1},
\end{aligned}
\end{equation}
where we used the fact that $\lambda_{j\e_i} = \lambda_{j\e_1}$, by the 
symmetry of the eigenvalues (see \eqref{eigen-rel}).

In order to complete the proof we recall Lemma 4.1 from \cite{muel-neum-2021}. 

\begin{lemma}\label{lemma-coef}  There exists a constant $C>0$ not depending on $N$ and $\beta$ such that, 
\begin{equation*}
  \sum_{j \in S_N^1 \setminus \{0\}}\alpha_{j}^2\lambda_{j\e_1} \leq CN^d. 
\end{equation*}
\end{lemma} 
 From Lemma \ref{lemma-coef} and \eqref{i2-fin} we conclude that 
\begin{equation}  
I_{2,d,D} \leq  C\beta a^2  N^d, 
\end{equation}
 which completes the proof of Proposition \ref{prop-bnd-I} part (iii).
\end{proof}

 \section{Large distance tail estimates} \label{sec-large-t}  

Assume first that $d=2$ and $d=1$.  Let $\alpha>0$, then from \eqref{eq:def-Q} we have 
\begin{equation}  \label{q-n-k} 
\begin{aligned} 
&\log Q_N (R_{N,2,1} >  \alpha  N^{4/3}(\log N)^{4/3}) \\
&\leq \log P(R_{N,2,1}>\alpha  N^{4/3}(\log N)^{4/3}) -\log Z_{N,2,1}. 
\end{aligned} 
\end{equation} 
From the proof in Section 5 of \cite{muel-neum-2021}, which uses  standard Gaussian estimates, it follows that for any constant $\kappa (N)>0$ we have  
\begin{equation*}
P(R_{N,2,1}>   \alpha  \beta^{-1/2}  \kappa(N) )  \leq  C_1 \exp\left\{-  c_2  \kappa(N)^2(\log N)^{-2}\right\}. 
\end{equation*} 
where $C_1, c_2>0$ are constants not depending on $N$. 

We therefore get, 
\begin{equation*}
P(R_{N,2,1}>   \alpha  \beta^{-1/2}   N^{4/3}(\log N)^{4/3} )  \leq  C_1 \exp\left\{- c_2 \alpha^2N^{8/3} (\log N)^{2/3}  \right\}. 
\end{equation*}  
 Using this bound together with Proposition \ref{prop-z}(i) and \eqref{q-n-k} we get for all $\alpha \geq 1$, 
\begin{align*} 
&\log Q_N (R_{N,2,1}>  \alpha \beta^{-1/2}  (\beta+\gamma)^{1/2} N^{4/3}(\log N)^{4/3}) \\
 & \leq \log P(R_{N,2,1}> \alpha\beta^{-1/2} (\beta+\gamma)^{1/2} N^{4/3}(\log N)^{4/3}) -\log Z_{N,2,1}  \\
 &\leq- (\beta+\gamma)  N^{8/3} (\log N)^{2/3} \big( c_3    \alpha^2  -c_4  \big). 
 \end{align*} 
We then can choose $\alpha$ to be large enough to get the large distance tail estimate in Theorem \ref{th:main}. 

The proof for $d\geq 3$ and $1\leq D \leq d$ follows similar lines, only now we use the bound 
\begin{equation*}
P(R_{N,d,D}>   \alpha  \beta^{-1/2}  \kappa(N) )  \leq  C_1 \exp\left\{-  c_2  \kappa(N)^2 \right\}.
\end{equation*} 
Note that the $\log$-correction does not appear on the right-hand side due to Proposition \ref{lem-cov}(ii). 

Together with Proposition \ref{prop-z}(ii) we get 
\begin{align*} 
&\log Q_N (R_{N,d,D}>  \alpha \beta^{-1/2}  (\beta+\gamma)^{1/2} N^{\frac{d}{2}+\frac{d-D}{D+2}} ) \\
 & \leq \log P(R_{N,d,D}> \alpha\beta^{-1/2} (\beta+\gamma)^{1/2}  N^{\frac{d}{2}+\frac{d-D}{D+2}} ) -\log Z_{N}  \\
 &\leq- (\beta+\gamma)  N^{d+\frac{2(d-D)}{D+2}}  \big( c_3    \alpha^2  -c_4  \big). 
 \end{align*} 
We then can choose $\alpha$ to be large enough to get the large distance tail estimate in Theorem \ref{th:main}. 

\section{Small distance tail estimates} \label{sec-small} 
Let $\eps>0$, and $\kappa (N)>0$ to be specified later. Then from \eqref{eq:def-Q} we have the following: 
\begin{equation}  \label{d1} 
\begin{split} 
\log &Q_N (R_{N,d,D} <  \eps \kappa(N)) \\
& \leq \log E\left[\exp\left\{-\gamma \int_{ \mathbb{R}^D}\ell_N(\y)^2d\y\right\}\mathbf{1}_{\{R_{N,d,D} <\eps \kappa (N) \}} \right] 
      -\log Z_{N,d,D}.  
\end{split}  
\end{equation} 
Let 
\begin{equation} \label{j-def} 
\tilde J_{N,d,D}:= E\left[\exp\left\{-\gamma\int_{\mathbb{R}^D}\ell_N(\y)^2d\y\mathbf{1}_{\{R_{N,d,D} <\eps \kappa (N) \}}\right\} \right]. 
\end{equation} 
Note that on $\{R_{N,d,D} <\eps\kappa (N) \}$ we have 
\begin{equation} \label{gh1} 
\begin{aligned} 
 \int_{\mathbb{R}^D}\ell_N(\y)^2d\y 
&= 2^d\kappa(N)^{D} \eps^D \int_{[-\eps \kappa(N),\eps \kappa(N)]^{D}}\ell_N(\y)^2 \frac{1}{2^D \kappa(N)^D\eps^D}d\y\\
&\geq 2^D\kappa(N)^D \eps^D \left(\int_{[-\eps \kappa(N),\eps \kappa(N)]^{D}}
\ell_N(\y) \frac{1}{2^D \kappa(N)^D\eps^D}d\y\right)^2 \\
&= \frac{1}{2^D \kappa(N)^D \eps^D} \left(\int_{[-\eps \kappa(N),\eps \kappa(N)]^{D}}\ell_N(\y) d\y\right)^2, 
\end{aligned} 
\end{equation} 
where we used Jensen's inequality. 
Since on $\{R_{N,d,D} <\eps \kappa(N) \}$ we have 
\begin{equation*}
\int_{[-\eps \kappa(N),\eps \kappa(N)]^{D}}\ell_N(\y) d\y= |S_N|= (2N+1)^d, 
\end{equation*}
together with \eqref{gh1} we get that  
\begin{equation} \label{l-lower} 
\int_{\mathbb{R}^D}\ell_N(\y)^2d\y  \geq  \frac{2^{2d-D}N^{2d}\kappa(N)^{-D}}{ \eps^D}.
\end{equation} 
From \eqref{j-def} and \eqref{l-lower} we have 
\begin{equation}  \label{small}
\tilde J \leq \exp\left( -\gamma \frac{2^{2d-D}N^{2d}\kappa(N)^{-D}}{ \eps^D}\right). 
\end{equation}
From \eqref{d1}, \eqref{j-def}, \eqref{small} and Proposition \ref{prop-z}(i) we get for $d=2$ and D=1 (taking  $\kappa(N)=\gamma (\beta +\gamma)^{-1}  N^{4/3}( \log N)^{-2/3}$): 
\begin{align*}
\log &  Q_N (R_{N,2,1} < \eps \gamma (\beta +\gamma)^{-1}  N^{4/3}( \log N)^{-2/3} ) \\
 & \leq -  (\beta +\gamma) \left(   \frac{8 N^{8/3}(\log N)^{2/3}}{ \eps} -  C N^{8/3}(\log N)^{2/3} \right). 
\end{align*} 
By choosing $\eps>0$ small enough it follows that  
\begin{equation*}
\lim_{N\rightarrow \infty}  \log  Q_N (R_{N,2,1} < \eps \gamma (\beta +\gamma)^{-1} N^{4/3}( \log N)^{-2/3}) =0. 
\end{equation*}
Repeating the same steps as in the case where $d=2$ and $D=1$, plugging in   
$$ 
\kappa(N)=\gamma^{1/D} (\beta +\gamma)^{-1/D} N^{\frac{1}{D} \left(d-\frac{2(d-D)}{d-D+2} \right)}
$$ 
to \eqref{small} we get, 
$$
\tilde J \leq e^{ -(\beta +\gamma)\frac{2^{2d-D}N^{d+\frac{2(d-D)}{d-D+2}}}{ \eps^D}}. 
$$
Together with Proposition \ref{prop-z}(ii) this give the following bound for $d\geq 3$ and $1\leq D \leq d$,
\begin{equation*}
\begin{aligned}
& \log  Q_N \left(R_{N,d,D} < \eps \gamma^{1/D} (\beta +\gamma)^{-1/D} N^{\frac{1}{D} \left(d-\frac{2(d-D)}{D+2} \right)} \right)  \\
 &  \leq -(\beta +\gamma) \left(   \frac{N^{d+\frac{2(d-D)}{D+2 }}}{ \eps^D} -  C N^{d+\frac{2(d-D)}{D+2 }}  \right).
 \end{aligned} 
\end{equation*}
Then choosing $\eps>0$ sufficiently small and taking the limit where 
$N\rightarrow \infty$ completes the proof of Theorem \ref{th:main}.


\end{document}